\documentclass{article}
\usepackage[utf8]{inputenc}

\usepackage{amsthm}
\usepackage{amsmath}
\usepackage{amssymb}
\usepackage{mathtools}
\usepackage{varioref}
\usepackage{booktabs}
\usepackage[margin=1.25in]{geometry}
\usepackage{listings}
\usepackage{paralist}
\usepackage{graphicx}
\usepackage{xcolor}
\usepackage{comment}
\usepackage{afterpage}
\usepackage{bbm}
\usepackage{cancel}
\usepackage{enumerate}
\usepackage{enumitem}
\usepackage{parallel}
\usepackage{algorithm,algorithmic}
\usepackage{xcolor}
\usepackage{enumerate}
\usepackage{caption}
\usepackage{subcaption}
\usepackage{tikz}
\usetikzlibrary{arrows,automata,positioning}

\usepackage[affil-it]{authblk}

\usepackage{hyperref} 					
\usepackage{setspace} 					
\usepackage[round]{natbib}

\DeclareMathOperator{\relint}{relint}
\DeclareMathOperator{\relbd}{relbd}
\DeclareMathOperator{\conv}{conv}

\hypersetup{
	colorlinks=true, 
	linkcolor={blue}, 
	linkbordercolor={white},
	urlcolor={blue}, 
	citecolor={blue}
}

\newcommand{\E}{\operatorname{E}}

\newtheorem{theorem}{Theorem}
\newtheorem{proposition}{Proposition}
\newtheorem{lemma}{Lemma}
\newtheorem{definition}{Definition}
\newtheorem{ex}{Example}

\allowdisplaybreaks
\numberwithin{equation}{section}

\title{Existence of Direct Density Ratio Estimators}
\author{
Erika Banzato\thanks{Department of Statistical Sciences, University of Padova; \url{erika.banzato@unipd.it}},~~
Mathias Drton\thanks{TUM School of Computation, Information and Technology and Munich Data Science Institute, Technical University of Munich; Munich Center for Machine Learning; \url{mathias.drton@tum.de}},~~
Kian Saraf-Poor\thanks{Department of Statistics, Columbia University; \url{ks4291@columbia.edu}},
~~and~
Hongjian Shi\thanks{TUM School of Computation, Information and Technology, Technical University of Munich; \url{hongjian.shi@tum.de}}
}

\date{}

\begin{document}

\setlength{\abovedisplayskip}{5pt}
\setlength{\belowdisplayskip}{5pt}
\setlength{\abovedisplayshortskip}{5pt}
\setlength{\belowdisplayshortskip}{5pt}

\maketitle
\begin{abstract}
    Many two-sample problems call for a comparison of two distributions from an exponential family.  Density ratio estimation methods provide ways to solve such problems through direct estimation of the differences in natural parameters.  The term direct indicates that one avoids estimating both marginal distributions.  In this context, we consider the Kullback--Leibler Importance Estimation Procedure (KLIEP), which has been the subject of recent work on differential networks.  Our main result shows that the existence of the KLIEP estimator is characterized by whether the average sufficient statistic for one sample belongs to the convex hull of the set of all sufficient statistics for data points in the second sample.  
    For high-dimensional problems it is customary to regularize the KLIEP loss by adding the product of a tuning parameter and a norm of the vector of parameter differences.  We show that the existence of the regularized KLIEP estimator requires the tuning parameter to be no less than the dual norm-based distance between the average sufficient statistic and the convex hull.  The implications of these existence issues are explored in applications to differential network analysis.   
\end{abstract}

\section{Introduction}
\label{sec:intro}

Let $\boldsymbol{X}=(\boldsymbol{X}_1,\dots,\boldsymbol{X}_{n_x})$ be an i.i.d.~sample drawn from a distribution $P$.  Let $\boldsymbol{Y}=(\boldsymbol{Y}_1,\dots,\boldsymbol{Y}_{n_y})$ be an independent second sample drawn from a distribution $Q$.  The resulting two-sample problem asks for statistical inference of differences between $P$ and $Q$.
An obvious route to solving this problem is to first separately estimate $P$ and $Q$ and then statistically inspect differences in the estimates.  However, in high-dimensional problems, even simple statistical models for $P$ and $Q$ may feature a large number of parameters, and estimation of $P$ and $Q$ may become difficult.

One prominent application giving rise to high-dimensional two-sample problems is differential network analysis.  Such an analysis infers differences between two graphical models \citep{doi:10.1146/annurev-statistics-060116-053803} in contexts such as studying a gene regulatory network under two experimental conditions \citep{MR4218944}.  In this setting, even the simplest pairwise interaction models are parametrized by interaction matrices that, for $m$-dimensional observations, hold on the order of $m^2$ many parameters, which results in challenging problems for higher dimension $m$.

In differential network analysis with Gaussian graphical models, the interaction matrix corresponds to the inverse covariance matrix.  Exploiting the matrix inversion relation between the interactions and the easily estimable covariances, \citet{MR3215346} shows how to directly estimate the difference between two inverse covariance matrices.  Their method avoids the estimation of the many nuisance parameters that encode the shared structure and is, thus, statistically feasible for high-dimensional problems provided mere sparsity of the differences.  This approach has been further developed by \citet{zhang:zou:2014}, \citet{MR3371002}, and \citet{yuan:xi:chen:deng:2017} but is restricted to the Gaussian setting.  

For more general settings involving exponential families, a framework to solve two-sample problems emerges from the literature on density ratio estimation \citep{MR2895762}.  This is particularly promising for differential network analysis with general graphical models of exponential families as it also allows one to overcome difficulties due to the fact that such general models typically feature an intractable normalizing constant.  The work of \citet{MR3234638} implements these ideas, and their KLIEP (Kullback--Leibler Importance Estimation Procedure) algorithm was later extended to an $\ell_1$-regularized version for high dimensions \citep{MR3662445}.  Moreover, \citet{MR4349123} developed a bootstrap-based method for formal statistical inference in a high-dimensional context.

In this paper, we study existence issues for the KLIEP estimator, without or with regularization.  The KLIEP loss is built by estimating the inner product between the difference vector and the sufficient statistic for the considered exponential family from one sample, say $\boldsymbol{X}$, and estimating a ratio of normalizing constants from the other sample $\boldsymbol{Y}$.
If the considered exponential family model has sufficient statistic $\mathbf{t}$, then the KLIEP loss depends on the data through the average sufficient statistic in the $\boldsymbol{X}$-sample,
\begin{equation}\label{eq:avetx}
\bar{\boldsymbol{t}}^x = \frac{1}{n_x}\sum_{i=1}^{n_x} \mathbf{t}(\boldsymbol{X_i}),
\end{equation}
as well as the set of all sufficient statistics for data points in the $\boldsymbol{Y}$-sample,
\begin{equation}\label{eq:setty}
\mathcal{T}^y = \left\{ \mathbf{t}(\boldsymbol{Y}_j): j=1,\dots,n_y\right\}.
\end{equation}
Taking the convex hull of the latter gives the convex polytope 
\begin{equation}\label{eq:Py}
\boldsymbol{C}^y =\,\text{conv}(\mathcal{T}^y ).
\end{equation}
In the unregularized case, we show that the existence of the KLIEP estimator is characterized by a polyhedral condition, namely, whether $\bar{\boldsymbol{t}}^x$ belongs to the polytope $\boldsymbol{C}^y$.   When regularizing, the KLIEP loss is modified by adding the product of a tuning parameter and a norm of the difference vector.  We show that the existence of the regularized KLIEP estimator requires the tuning parameter to be no less than the dual norm-based distance between the average sufficient statistic $\bar{\boldsymbol{t}}^x$ and the polytope~$\boldsymbol{C}^y$.  

The existence results are developed in Section~\ref{sec:exist}, following a review of preliminaries in Section~\ref{sec:preliminaries}.  In Section~\ref{sec:diffnets}, we take up a specific two-sample problem, namely,  differential network analysis with Gaussian graphical models, 
for which we explore the implications of existence issues in $\ell_1$-regularized KLIEP estimation.  Our experiments indicate that tuning parameter values considered in theoretical analysis of regularized KLIEP \citep{MR3662445} can be too small to ensure existence.  Therefore, we conclude the paper with a discussion in Section~\ref{sec:dis} on the potential use of an elastic net-type penalty to ensure existence.

\section{Density ratio estimation and the KLIEP loss}
\label{sec:preliminaries}

Let $\boldsymbol{X}=(\boldsymbol{X}_1,\dots,\boldsymbol{X}_{n_x})$ and $\boldsymbol{Y}=(\boldsymbol{Y}_1,\dots,\boldsymbol{Y}_{n_y})$ be two independent samples  drawn i.i.d.~from probability distributions $P$ and $Q$, respectively.  Assume $P$ and $Q$ belong to an exponential family with densities of the form
\begin{equation}
\label{eq:expfam}
    f(\boldsymbol{x};\boldsymbol{\theta}) = \frac{1}{Z(\boldsymbol{\theta})} 
    \exp\left\{
    \sum_{v=1}^{k} {\theta}_{v} \,{t}_{v}(\boldsymbol{x}) \right\},
\end{equation}
where $\mathbf{t}(\boldsymbol{x})=({t}_{1}(\boldsymbol{x}),\dots,{t}_{k}(\boldsymbol{x}))^{\top}$ is a $k$-dimensional sufficient statistic, $\boldsymbol{\theta}=(\theta_1,\dots,\theta_k)^{\top}$ is the associated natural parameter vector, and $Z(\boldsymbol{\theta})$ is the normalizing constant.  We write $\nu$ for the dominating measure of the exponential family.

If $P$ has density $p(\boldsymbol{x})=f(\boldsymbol{x};\boldsymbol{\theta}^{(p)})$ and $Q$ has density $q(\boldsymbol{x})=f(\boldsymbol{x};\boldsymbol{\theta}^{(q)})$, both with respect to the dominating measure $\nu$, then the parameter of interest is the difference vector
\begin{equation}
    \label{eq:difference}
    \boldsymbol{\Delta} = \boldsymbol{\theta}^{(p)}-\boldsymbol{\theta}^{(q)}.
\end{equation}
The difference vector emerges when considering the density ratio
\begin{equation}
    \label{eq:density-ratio}
       \frac{p(\boldsymbol{x})}{q(\boldsymbol{x})}\equiv\frac{ f(\boldsymbol{x};\boldsymbol{\theta}^{(p)})}{ f(\boldsymbol{x};\boldsymbol{\theta}^{(q)})} \;=\;
       \frac{Z(\boldsymbol{\theta}^{(q)})}{Z(\boldsymbol{\theta}^{(p)})}
    \exp\left\{\sum_{v=1}^{k} (\theta_{v}^{(p)}-{\theta}_{v}^{(q)}) \,\mathbf{t}_{v}(x_v)\right\} \propto
    \exp\left\{\boldsymbol{\Delta}^{\top} \, \mathbf{t}(\boldsymbol{x})\right\}.
\end{equation}
Since the density ratio integrates to one with respect to the measure $q(\boldsymbol{x})\mathrm{d}\nu(\boldsymbol{x})$, it follows that 
\begin{equation}
 \label{eq:r}
    \frac{p(\boldsymbol{x})}{q(\boldsymbol{x})}
    =\frac{1}{N(\boldsymbol{\Delta};q)}
    \exp\left\{\boldsymbol{\Delta}^{\top}  \textbf{t}(\boldsymbol{x})\right\}=:r_{\boldsymbol{\Delta},q}(\boldsymbol{x}),
\end{equation}
where 
\begin{align}
    N(\boldsymbol{\Delta}; q) &:= 
    \int q(\boldsymbol{x}) \exp\left\{\boldsymbol{\Delta}^{\top}  \textbf{t}(\boldsymbol{x})\right\} \mathrm{d}\nu(\boldsymbol{x}).\label{normalizer}
\end{align}
Rephrasing~\eqref{eq:r}, the true difference vector for $P$ and $Q$ is characterized as the vector $\boldsymbol{\Delta}$ for which  the following equality among densities holds:
\begin{equation}
\label{eq:p=rq}
    p(\boldsymbol{x})= r_{\boldsymbol{\Delta},q}(\boldsymbol{x})\cdot q(\boldsymbol{x}).
\end{equation}

Write $D_\mathrm{KL}(p_1\parallel p_2)$ for the Kullback--Leibler (KL) divergence of a distribution with density $p_2$ from another distribution with density $p_1$.  It is well-known that the KL divergence is nonnegative and vanishes if and only if the two distributions are equal.   
Hence, by Equation~\eqref{eq:p=rq}, the true difference vector~$\boldsymbol{\Delta}$  in our two-sample problem is characterized as the unique minimizer of the function
\begin{align}
\boldsymbol{\Delta} \mapsto
    D_\mathrm{KL}(p\parallel r_{\boldsymbol{\Delta},q}\, q)
    & = \int p(\boldsymbol x) \log \frac{p(\boldsymbol x)}{q(\boldsymbol x)r_{\boldsymbol{\Delta},q}(\boldsymbol{x})}\mathrm{d}\nu(\boldsymbol{x}) \notag\\
    & = -\int p(\boldsymbol x) \log r_{\boldsymbol{\Delta},q}(\boldsymbol{x})\mathrm{d}\nu(\boldsymbol{x}) +\text{const}.\label{eq:DKL}
\end{align}
From definitions in \eqref{eq:r} and~\eqref{normalizer}, we obtain the following fact, which is the basis for the \textit{Kullback--Leibler importance estimation procedure} (KLIEP).

\begin{proposition}
\label{prop:KLIEP-loss}
Let $p(\boldsymbol{x})=f(\boldsymbol{x};\boldsymbol{\theta}^{(p)})$ and $q(\boldsymbol{x})=f(\boldsymbol{x};\boldsymbol{\theta}^{(q)})$ be two densities from the considered exponential family.  Then the difference vector $\boldsymbol{\Delta} = \boldsymbol{\theta}^{(p)}-\boldsymbol{\theta}^{(q)}$ is the unique minimizer of the \emph{population KLIEP loss}
\begin{equation}
\ell_\mathrm{KL}(\boldsymbol{\Delta}; p,q)=-\boldsymbol{\Delta}^{\top} \E_p[ \mathbf{t}(\boldsymbol{X}_1)] + \log \E_q[\exp\{\boldsymbol{\Delta}^{\top}  \mathbf{t}(\boldsymbol{Y}_1)\}],
\end{equation}
where $\E_p$ and $\E_q$ are expectations, and subscripts $p$ and $q$ emphasize that $\boldsymbol{X}_1\sim P$ and $\boldsymbol{Y}_1\sim Q$.
\end{proposition}

\begin{proof}[Proof of Proposition~\ref{prop:KLIEP-loss}]
We first prove the KLIEP loss is equal to the first term in \eqref{eq:DKL}. It holds~that
\begin{align*}
-\int p(\boldsymbol x) \log r_{\boldsymbol{\Delta},q}(\boldsymbol{x})\mathrm{d}\nu(\boldsymbol{x})
&=-\int p(\boldsymbol x) \log \frac{\exp\left\{\boldsymbol{\Delta}^{\top}  \textbf{t}(\boldsymbol{x})\right\}}{N(\boldsymbol{\Delta}; q)} \mathrm{d}\nu(\boldsymbol{x}) \\
&=-\int p(\boldsymbol x)  \boldsymbol{\Delta}^{\top}  \textbf{t}(\boldsymbol{x}) \mathrm{d}\nu(\boldsymbol{x}) + \int p(\boldsymbol x) \log N(\boldsymbol{\Delta}; q) \mathrm{d}\nu(\boldsymbol{x}) \\
&=-\boldsymbol{\Delta}^{\top} \int p(\boldsymbol x) \textbf{t}(\boldsymbol{x}) \mathrm{d}\nu(\boldsymbol{x}) + \log N(\boldsymbol{\Delta}; q) \\
&=-\boldsymbol{\Delta}^{\top} \int p(\boldsymbol{x}) \mathbf{t}(\boldsymbol{x}) \mathrm{d}\nu(\boldsymbol{x})
  +\log \int q(\boldsymbol{x}) \exp\{\boldsymbol{\Delta}^{\top} \mathbf{t}(\boldsymbol{x})\} \mathrm{d}\nu(\boldsymbol{x})\\
&=-\boldsymbol{\Delta}^{\top} \E_p[ \mathbf{t}(\boldsymbol{X}_1)] + \log \E_q[\exp\{\boldsymbol{\Delta}^{\top}  \mathbf{t}(\boldsymbol{Y}_1)\}] \\
&=\ell_\mathrm{KL}(\boldsymbol{\Delta}; p,q).
\end{align*}
The rest follows from the fact the true difference vector is the unique minimizer of the function $\boldsymbol{\Delta} \mapsto
    D_\mathrm{KL}(p\parallel r_{\boldsymbol{\Delta},q}\, q)$.
\end{proof}

For practical estimation, we may replace population expectations by sample averages.

\begin{definition}[KLIEP] 
Given two samples $\boldsymbol{X}=(\boldsymbol{X}_1,\dots,\boldsymbol{X}_{n_x})$ and $\boldsymbol{Y}=(\boldsymbol{Y}_1,\dots,\boldsymbol{Y}_{n_y})$, the (empirical) KLIEP loss is defined as
\begin{align}
    \ell_\mathrm{KL}(\boldsymbol{\Delta})&\;\equiv\; \ell_\mathrm{KL}(\boldsymbol{\Delta}; \boldsymbol{X}, \boldsymbol{Y}) 
    \label{eq:kliep}
               \;=\; 
           -\frac{1}{n_x} \sum_{i=1}^{n_x} \boldsymbol{\Delta}^{\top}  \mathbf{t}(\boldsymbol{X}_i) +
           \log\left[ \frac{1}{n_y} \sum_{j=1}^{n_y}\exp\left\{\boldsymbol{\Delta}^{\top}  \mathbf{t}(\boldsymbol{Y}_j)\right\}\right].
\end{align}
The KLIEP estimator $\hat{\boldsymbol{\Delta}}$ is the minimizer of $\ell_\mathrm{KL}(\boldsymbol{\Delta})$.
\end{definition}

Using results from \citet[page~74]{MR2061575}, it is easily seen that the KLIEP loss $\ell_\mathrm{KL}$ is a convex function of $\boldsymbol{\Delta}$, and a global minimizer can be found using standard optimization techniques.  However, as we characterize in this paper, such a minimizer need not exist.  Indeed, the loss may be unbounded or, in special situations, bounded but with the infimum not attained.  Our characterization of when such inexistence occurs will reflect the fact that the KLIEP loss treats the two samples asymmetrically; see also the discussion in \citet{MR3662445}.

\section{Existence of KLIEP estimators}
\label{sec:exist}

In this section, we characterize for which datasets a minimizer of the KLIEP loss $\ell_\mathrm{KL}$ exists.  
Given two samples $\boldsymbol{X}=(\boldsymbol{X}_1,\dots,\boldsymbol{X}_{n_x})$ and $\boldsymbol{Y}=(\boldsymbol{Y}_1,\dots,\boldsymbol{Y}_{n_y})$, $\bar{\boldsymbol{t}}^{x}$ and $\mathcal{T}^y$ defined respectively by Equations~\eqref{eq:avetx} and~\eqref{eq:setty} will be the sufficient statistics for (\ref{eq:kliep}), where $\bar{\boldsymbol{t}}^{x}$ and $\mathcal{T}^y$ are obtained from the samples $\boldsymbol{X}$ and $\boldsymbol{Y}$, respectively. 
To ease the notation, we write $\bar{\boldsymbol{t}}^x=(\bar{t}^{x}_1,\dots,\bar{t}^{x}_k)^{\top}$ with each entry defined as 
\begin{equation}\label{eq:defR}
\bar{t}^{x}_v:=\frac{1}{n_x} \sum_{i=1}^{n_x} {t}_{v}(\boldsymbol{X}_i), \quad v=1,\dots,k,
\end{equation}
and write
$\mathbf{t}(\boldsymbol{Y}_j)\equiv \boldsymbol{t}^{y}_j = (t_{1j}^{y},\dots,t_{kj}^{y})^{\top}$ with each entry given by 
\begin{equation}\label{eq:defT}
t^{y}_{vj}:= {t}_v(\boldsymbol{Y}_j), \quad v=1,\dots,k,~j=1,\dots,n_y.
\end{equation}
In the following, for a set $S\subseteq \mathbb{R}^k$, 
let $\mathrm{span}(S)$ denote the spanned space of $S$,
let $\relint(S)$ denote the relative interior of $S$,
and let $\relbd(S)$ denote the relative boundary of $S$;
see, e.g., Chapter 6 of \citet{MR0274683} for a comprehensive introduction to relative interior and relative boundary.

\subsection{KLIEP estimators}

The following theorem summarizes how the existence of the KLIEP estimator is determined by a polyhedral condition, specifically, whether $\bar{\boldsymbol{t}}^x$ belongs to the relative interior of the polytope~$\boldsymbol{C}^y$ or not.  
The KLIEP estimator, as a minimizer of the KLIEP loss $\ell_\mathrm{KL}$, does not exist when~$\bar{\boldsymbol{t}}^x$ lies on the relative boundary or outside of~$\boldsymbol{C}^y$; the latter will happen with high probability in the high-dimensional case.  
This also motivates us to consider the regularized KLIEP estimators, which will be discussed in the next section.

\begin{theorem} \label{theorem:1}
Given two samples $\boldsymbol{X}=(\boldsymbol{X}_1,\dots,\boldsymbol{X}_{n_x})$ and $\boldsymbol{Y}=(\boldsymbol{Y}_1,\dots,\boldsymbol{Y}_{n_y})$,
let $\bar{\boldsymbol{t}}^x$ and $\boldsymbol{C}^{y}$ be defined by Equations~\eqref{eq:avetx} and \eqref{eq:Py}, respectively.  
Then it holds that
\begin{enumerate}[itemsep=-.25ex, label=(\roman*)]
	\item $\bar{\boldsymbol{t}}^x \in \relint(\boldsymbol{C}^{y})$ if and only if $\ell_\mathrm{KL}$ achieves a global minimum;
	\item $\bar{\boldsymbol{t}}^x \in \relbd(\boldsymbol{C}^{y})$ if and only if $\ell_\mathrm{KL}$ is bounded from below but does not attain a global minimum;
	\item $\bar{\boldsymbol{t}}^x \notin \boldsymbol{C}^{y}$ if and only if $\ell_\mathrm{KL}$ is unbounded from below.
\end{enumerate}
\end{theorem}

Before providing a strict proof, we first demonstrate how the KLIEP loss behaves (the ``only if'' part) under scenarios (i), (ii), and (iii) using the following toy example in dimension one.

\begin{ex}
When $\mathcal{T}^y=\{-1,0,1,2\}$, consider $\bar{\boldsymbol{t}}^x=1,2,3$ (where the conditions in scenarios (i), (ii), and (iii) are satisfied, respectively).
The KLIEP loss $\ell_\mathrm{KL}$ with respect to the difference vector parameter $\boldsymbol{\Delta}$ is portrayed in Figure~\ref{fig:toy}. We notice that the KLIEP loss behaves as stated in Theorem~\ref{theorem:1}. 
\end{ex}

\begin{figure}[t]
\captionsetup[subfigure]{labelformat=empty}
    \begin{subfigure}{0.33\textwidth}
    \centering
    \includegraphics[width=\textwidth]{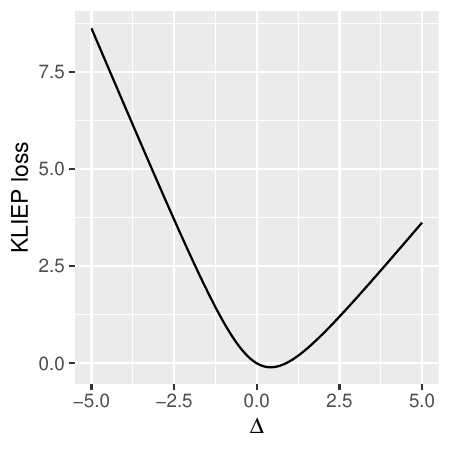}
    \caption{(i) $\bar{\boldsymbol{t}}^x=1 \in \relint(\boldsymbol{C}^{y})$}
    \end{subfigure}\hfill
    \begin{subfigure}{0.33\textwidth}
    \centering
    \includegraphics[width=\textwidth]{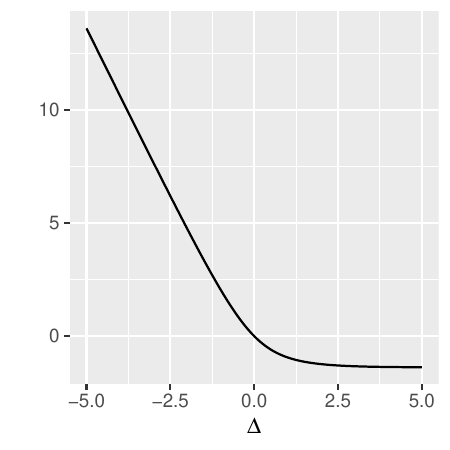}
    \caption{(ii) $\bar{\boldsymbol{t}}^x=2 \in \relbd(\boldsymbol{C}^{y})$}
    \end{subfigure}\hfill
    \begin{subfigure}{0.33\textwidth}
    \includegraphics[width=\textwidth]{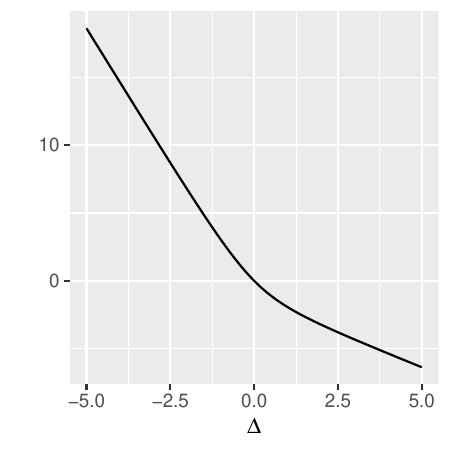}
    \caption{(iii) $\bar{\boldsymbol{t}}^x=3 \notin \boldsymbol{C}^{y}$}
    \centering
    \end{subfigure}
    \caption{{The KLIEP loss function when $\bar{\boldsymbol{t}}^{x}$ lies in the relative interior (when~$\bar{\boldsymbol{t}}^{x}=1$), on the relative boundary (when~$\bar{\boldsymbol{t}}^{x}=2$), and outside (when~$\bar{\boldsymbol{t}}^{x}=3$) of the convex hull of $\mathcal{T}^y=\{-1,0,1,2\}$.}}
    \label{fig:toy}
\end{figure}

\begin{proof}[Proof of Theorem~\ref{theorem:1}]
We begin with assertion (i).  Recall that the convex hull $\boldsymbol{C}^{y}$ of the set $\mathcal{T}^{y}$ can be represented by
\[
\boldsymbol{C}^{y} = \left\{\sum_{j=1}^{n_y}\alpha_j \boldsymbol{t}^{y}_j: \alpha_j\ge 0,~\sum_{j=1}^{n_y}\alpha_j=1\right\}
\]
and the relative interior of the convex hull $\boldsymbol{C}^{y}$ can be represented by
\[
\relint(\boldsymbol{C}^{y}) = \left\{\sum_{j=1}^{n_y}\alpha_j \boldsymbol{t}^{y}_j: \alpha_j > 0,~\sum_{j=1}^{n_y}\alpha_j=1\right\};
\]
see \cite[page~155]{MR0780745}.  
First, we show that $\boldsymbol{R} \in \relint(\boldsymbol{C})$ if $\ell_\mathrm{KL}$ achieves a global minimum at $\boldsymbol{\Delta}_0$.  Let us expand $\ell_\mathrm{KL}$ as
\begin{align}
    \ell_\mathrm{KL}(\boldsymbol{\Delta}) &= -\boldsymbol{\Delta}^{\top}  \bar{\boldsymbol{t}}^x + \log\left(\frac{1}{n_y}\sum_{j=1}^{n_y} \exp\{\boldsymbol{\Delta}^{\top}  \boldsymbol{t}^{y}_j\}\right) \notag\\
    &=-\sum_{v=1}^{k}\Delta_v \bar{t}^{x}_v + \log\left(\frac{1}{n_y}\sum_{j=1}^{n_y} \exp\left\{\sum_{v=1}^{k}\Delta_v t^{y}_{vj}\right\}\right).
\end{align}
Notice that a global minimum at $\boldsymbol{\Delta}_0$ is also a local minimum since $\mathbb{R}^k$, the domain of $\boldsymbol{\Delta}$, is an open set.
According to the interior extremum theorem, when the gradient $\nabla\ell_\mathrm{KL}(\boldsymbol{\Delta})$ exists, it must be the zero vector at a local minimum.
Here, the $v$-th component of the gradient $\nabla\ell_\mathrm{KL}(\boldsymbol{\Delta})$ is given~by
\begin{align}\label{eq:deriv}
    \frac{\partial\ell_\mathrm{KL}(\boldsymbol{\Delta})}{\partial \Delta_v} 
    &= -\,\bar{t}^{x}_v + \frac{\sum_{j=1}^{n_y} \exp\left\{\boldsymbol{\Delta}^{\top} \boldsymbol{t}^{y}_j\right\}t^{y}_{vj}}{\sum_{j=1}^{n_y} \exp\left\{\boldsymbol{\Delta}^{\top} \boldsymbol{t}^{y}_j \right\}} \notag\\
    &= -\,\bar{t}^{x}_v + \sum_{j=1}^{n_y} \frac{\exp\left\{\boldsymbol{\Delta}^{\top} \boldsymbol{t}^{y}_j\right\}}{\sum_{i=1}^{n_y} \exp\left\{\boldsymbol{\Delta}^{\top} \boldsymbol{t}^{y}_i\right\}} t^{y}_{vj} \notag\\
    &= -\,\bar{t}^{x}_v + \sum_{j=1}^{n_y} \alpha_j(\boldsymbol{\Delta}) t^{y}_{vj},\quad  v = 1,\dots,k,
\end{align}
where
\begin{align}
    \alpha_j(\boldsymbol{\Delta}) := \frac{\exp\left\{\boldsymbol{\Delta}^{\top} \boldsymbol{t}^{y}_j\right\}}{\sum_{i=1}^{n_y} \exp\left\{\boldsymbol{\Delta}^{\top} \boldsymbol{t}^{y}_i\right\}},\quad  j = 1,\dots,n_y.
\end{align}
\noindent Note that for any $\boldsymbol{\Delta}$, $\alpha_j(\boldsymbol{\Delta}) > 0$, $j =  1, \dots, n_y$ and $\sum_{j=1}^{n_y} \alpha_j(\boldsymbol{\Delta}) = 1$. 
Then at the global and local minimizer $\boldsymbol{\Delta}_0$, it holds that
\begin{equation}\label{eq:eqmin}
\begin{pmatrix}
\, 0\,  \\ \, 0\,  \\ \, \vdots\,  \\ \, 0\, 
\end{pmatrix}
=\nabla\ell_\mathrm{KL}(\boldsymbol{\Delta}_0)=
\begin{pmatrix}
-\bar{t}^{x}_1 + \sum_{j=1}^{n_y} \alpha_j(\boldsymbol{\Delta}_0) t^{y}_{1j} \\
-\bar{t}^{x}_2 + \sum_{j=1}^{n_y} \alpha_j(\boldsymbol{\Delta}_0) t^{y}_{2j} \\
\vdots \\
-\bar{t}^{x}_k + \sum_{j=1}^{n_y} \alpha_j(\boldsymbol{\Delta}_0) t^{y}_{kj} 
\end{pmatrix}.
\end{equation}
We deduce that
\[
\bar{t}^{x}_v = \sum_{j=1}^{n_y} \alpha_j(\boldsymbol{\Delta}_0) t^{y}_{vj},\quad  v = 1,\dots,k,
\]
and, thus,
\[
\bar{\boldsymbol{t}}^{x} = \sum_{j=1}^{n_y} \alpha_j(\boldsymbol{\Delta}_0) \boldsymbol{t}^{y}_j.
\]
Therefore, if $\ell_\mathrm{KL}$ achieves a global minimum, then $\bar{\boldsymbol{t}}^{x}$ must lie in the relative interior of the convex hull $\boldsymbol{C}^{y}$.

Now, we proceed with proving that $\bar{\boldsymbol{t}}^{x} \in \relint(\boldsymbol{C}^{y})$ implies $\ell_\mathrm{KL}$ achieves a global minimum.  Let~$\bar{\boldsymbol{t}}^{x}\in \relint(\boldsymbol{C}^{y})$, in other words, there exist $\alpha_1,\ldots,\alpha_{n_y} > 0$ such that $\sum_{j=1}^{n_y}\alpha_j = 1$ and $\bar{\boldsymbol{t}}^{x} = \sum_{j=1}^{n_y}\alpha_j \boldsymbol{t}^{y}_j$. To ease the notation, define $\boldsymbol{u}_j \coloneqq \boldsymbol{t}^{y}_j - \bar{\boldsymbol{t}}^{x}$, $j = 1, \ldots, n_y$. We can then write
\[
  \sum_{j=1}^{n_y}\alpha_j \boldsymbol{u}_j
= \sum_{j=1}^{n_y}\alpha_j (\boldsymbol{t}^{y}_j - \bar{\boldsymbol{t}}^{x})
= \sum_{j=1}^{n_y}\alpha_j \boldsymbol{t}^{y}_j - \bar{\boldsymbol{t}}^{x} = \boldsymbol{0}
\]
and
\begin{align*}
	\ell_\mathrm{KL}(\boldsymbol{\Delta}) 
	= -\boldsymbol{\Delta}^{\top}\bar{\boldsymbol{t}}^{x} + \log\left( \frac{1}{n_y} \sum_{j=1}^{n_y}\exp(\boldsymbol{\Delta}^{\top} \boldsymbol{t}^{y}_j)\right)
	=\log\left( \frac{1}{n_y} \sum_{j=1}^{n_y}\exp(\boldsymbol{\Delta}^{\top} \boldsymbol{u}_j)\right).
\end{align*}
Furthermore, define $\boldsymbol{U}$ to be the subspace of $\mathbb{R}^k$ spanned by $\boldsymbol{u}_1,\dots,\boldsymbol{u}_{n_y}$:
\begin{equation}\label{U}
    \boldsymbol{U} \coloneqq \mathrm{span}(\{\boldsymbol{u}_1,\ldots,\boldsymbol{u}_{n_y}\}),
\end{equation}
and define $\boldsymbol{U}^\perp$ to be the orthogonal complement of $\boldsymbol{U}$ in $\mathbb{R}^k$.
Since $\mathbb{R}^k=\boldsymbol{U}\oplus\boldsymbol{U}^\perp$, 
we can (uniquely) decompose any $\boldsymbol{\Delta}\in\mathbb{R}^k$ as $\boldsymbol{\Delta}=\boldsymbol{\Delta}_{\boldsymbol{U}} +\boldsymbol{\Delta}_{\boldsymbol{U}^\perp}$ with $\boldsymbol{\Delta}_{\boldsymbol{U}}\in\boldsymbol{U},\boldsymbol{\Delta}_{\boldsymbol{U}^\perp}\in\boldsymbol{U}^\perp$.  Notice that 
\begin{equation}\label{const}
\ell_\mathrm{KL}\left(\boldsymbol{\Delta}_{1} +\boldsymbol{\Delta}_{2}\right)=\ell_\mathrm{KL}\left(\boldsymbol{\Delta}_{1}\right),~~~
\boldsymbol{\Delta}_{1}\in\boldsymbol{U},~
\boldsymbol{\Delta}_{2}\in\boldsymbol{U}^\perp
\end{equation}
is constant for any fixed $\boldsymbol{\Delta}_{1}\in\boldsymbol{U}$. Hence, we may only consider the restriction
\begin{align*}
\ell_\mathrm{KL}\vert_{\boldsymbol{U}}:\;
& \boldsymbol{U}\to\mathbb{R},\\
& \boldsymbol{\Delta} \mapsto \ell_\mathrm{KL}(\boldsymbol{\Delta}).
\end{align*}
Without loss of generality, assume that $\boldsymbol{U}\ne\{\boldsymbol{0}\}$; otherwise, the existence of a global minimum would follow trivially. 
We will show that $\ell_\mathrm{KL}\vert_{\boldsymbol{U}}$ is a coercive function, i.e., for $\boldsymbol{\Delta}\in\boldsymbol{U}$, 
it holds that
\begin{equation}
\lim\limits_{\Vert\boldsymbol{\Delta}\Vert\to\infty}\ell_\mathrm{KL}\left(\boldsymbol{\Delta}\right)=+\infty.
\end{equation}
The next lemma is needed to prove that $\max_{1\le j\le n_y}{\boldsymbol{\Delta}}^{\top}\boldsymbol{u}_{j}$ attains a strictly positive minimum for all $\boldsymbol{\Delta}\in\boldsymbol{U}$ satisfying $\Vert\boldsymbol{\Delta}\Vert=1$, which serves as a key component of proving the coercivity of $\ell_\mathrm{KL}\vert_{\boldsymbol{U}}$.  
\begin{lemma}\label{lem:1}
The function 
\begin{align*}\label{fstar}
f^*:\;
&\mathbb{S}^{d-1}\cap\boldsymbol{U} \to \mathbb{R}, \\
&{\boldsymbol{\Delta}} \mapsto \max_{1\le j\le n_y}{\boldsymbol{\Delta}}^{\top}\boldsymbol{u}_{j}
\end{align*}
is continuous and strictly positive on the compact set $\mathbb{S}^{d-1}\cap\boldsymbol{U}$, 
where $\mathbb{S}^{d-1}:=\left\{\boldsymbol{\Delta}\in\mathbb{R}^k:\Vert\boldsymbol{\Delta}\Vert=1\right\}$ is the unit sphere in $\mathbb{R}^k$.  
\end{lemma}
\begin{proof}[Proof of Lemma~\ref{lem:1}]
As $f^*$ is a maximum of finite continuous functions ${\boldsymbol{\Delta}} \mapsto {\boldsymbol{\Delta}}^{\top}\boldsymbol{u}_{j}$, it is continuous as well. 

To show strict positivity, let ${\boldsymbol{\Delta}}\in\mathbb{S}^{d-1}\cap\boldsymbol{U}$. Since $\sum_{j=1}^{n_y}\alpha_j \boldsymbol{u}_j = \boldsymbol{0}$,
\[
\sum_{j=1}^{n_y}\alpha_j \boldsymbol{\Delta}^{\top}\boldsymbol{u}_j = 0.
\]
As $\boldsymbol{\Delta}$ can be represented by $\sum_{j=1}^{n_y}\beta_j \boldsymbol{u}_j$, $\boldsymbol{\Delta}^{\top}\boldsymbol{u}_j$ cannot all be zero for every $j$; otherwise, $\beta_j=0$ for every $j$, leading to $\boldsymbol{\Delta} = 0$, which is a contradiction. Therefore, there exists some $1\le j_0\le n_y$ such that ${\boldsymbol{\Delta}}^{\top} \boldsymbol{u}_{j_0} \ne 0$. If ${\boldsymbol{\Delta}}^{\top} \boldsymbol{u}_{j_0} > 0$, the claim has been proved.  If ${\boldsymbol{\Delta}}^{\top} \boldsymbol{u}_{j_0} < 0$, then there must exist some~$1\le j_1\le n_y$ such that ${\boldsymbol{\Delta}}^{\top} \boldsymbol{u}_{j_1} > 0$; otherwise, we can deduce that $\sum_{j=1}^{n_y}\alpha_j \boldsymbol{\Delta}^{\top}\boldsymbol{u}_j < 0$ recalling $\alpha_1,\ldots,\alpha_{n_y} > 0$  --- another contradiction.
\end{proof}
Lemma~\ref{lem:1} concludes that $f^*$ attains a strictly positive minimum on $\mathbb{S}^{d-1}\cap\boldsymbol{U}$:
\begin{equation}\label{alpha}
	\delta \coloneqq \min\limits_{{\boldsymbol{\Delta}}\in\mathbb{S}^{d-1}\cap\boldsymbol{U}} f^*({\boldsymbol{\Delta}})>0.
\end{equation}
We write
$
\tilde{\boldsymbol{\Delta}}\coloneqq {\boldsymbol{\Delta}}/{\Vert\boldsymbol{\Delta}\Vert}
$
for $\boldsymbol{\Delta}\in\boldsymbol{U}$ if $\boldsymbol{\Delta}\ne\boldsymbol{0}$. 
Then for $\boldsymbol{\Delta}\in\boldsymbol{U}$, when $\Vert\boldsymbol{\Delta}\Vert\to\infty$,
\begin{align}\label{coercive}
\lim\limits_{\Vert\boldsymbol{\Delta}\Vert\to\infty} \ell_\mathrm{KL}(\boldsymbol{\Delta})
&=\lim\limits_{\Vert\boldsymbol{\Delta}\Vert\to\infty}\log\left( \frac{1}{n_y} \sum_{j=1}^{n_y}\exp(\boldsymbol{\Delta}^{\top} \boldsymbol{u}_j)\right) \notag\\
&=\lim\limits_{\Vert\boldsymbol{\Delta}\Vert\to\infty}\log\left( \frac{1}{n_y} \sum_{j=1}^{n_y}\exp\left(\Vert\boldsymbol{\Delta}\Vert\tilde{\boldsymbol{\Delta}}^{\top} \boldsymbol{u}_j\right)\right) \notag\\
&\ge\lim\limits_{\Vert\boldsymbol{\Delta}\Vert\to\infty}\log\left( \frac{1}{n_y} \max_{1\le j\le n_y}\exp\left(\Vert\boldsymbol{\Delta}\Vert\tilde{\boldsymbol{\Delta}}^{\top} \boldsymbol{u}_{j}\right)\right) \notag\\
&=\lim\limits_{\Vert\boldsymbol{\Delta}\Vert\to\infty}\max_{1\le j\le n_y}\Vert\boldsymbol{\Delta}\Vert\tilde{\boldsymbol{\Delta}}^{\top} \boldsymbol{u}_{j}-\log(n_y) \notag\\
&\ge\lim\limits_{\Vert\boldsymbol{\Delta}\Vert\to\infty}\delta\Vert\boldsymbol{\Delta}\Vert-\log(n_y)=+\infty.
\end{align}
Hence $\ell_\mathrm{KL}\vert_{\boldsymbol{U}}$ is coercive. As a continuous and coercive function, $\ell_\mathrm{KL}\vert_{\boldsymbol{U}}$ attains a global minimum by the following lemma from \citet[page~543, Corollary~C.6]{MR2168305}.

\begin{lemma}\label{lem:coercive}
Let $f:\mathbb{R}^k\mapsto\mathbb{R}$ be a continuous function. Then the
function $f(x)$ has a global minimum if $f$ is coercive. 
\end{lemma}

Finally, using Equation~\eqref{const} we conclude that $\ell_\mathrm{KL}$ attains a global minimum.

\bigskip

Next, we prove assertion (iii).  First, we show that $\bar{\boldsymbol{t}}^{x} \not\in \boldsymbol{C}^{y}$ if $\ell_\mathrm{KL}$ is unbounded from below.  We prove it by contradiction.  Let $\bar{\boldsymbol{t}}^{x} \in \boldsymbol{C}^{y}$. There exist $\alpha_1,\ldots,\alpha_{n_y} \ge 0$ such that $\sum_{j=1}^{n_y}\alpha_j = 1$ and $\bar{\boldsymbol{t}}^{x}=\sum\limits_{j=1}^{n_y}\alpha_j \boldsymbol{t}^{y}_j$. 
For any $\boldsymbol{\Delta}\in\mathbb{R}^k$, consider a $j^*$ such that $\boldsymbol{\Delta}^{\top} \boldsymbol{t}^{y}_{j^*}=\max\limits_{1\le j\le n_y}\boldsymbol{\Delta}^{\top} \boldsymbol{t}^{y}_j$.  We deduce
\begin{align*}
	\ell_\mathrm{KL}(\boldsymbol{\Delta}) 
	&= -\boldsymbol{\Delta}^{\top} \bar{\boldsymbol{t}}^{x} + \log\left( \frac{1}{n_y} \sum_{j=1}^{n_y}\exp(\boldsymbol{\Delta}^{\top} \boldsymbol{t}^{y}_j)\right)\\
	&= -\boldsymbol{\Delta}^{\top} \sum\limits_{j=1}^{n_y}\alpha_j \boldsymbol{t}^{y}_j+\log\left( \frac{1}{n_y} \sum_{j=1}^{n_y}\exp(\boldsymbol{\Delta}^{\top} \boldsymbol{t}^{y}_j)\right) \\
	&\ge -\boldsymbol{\Delta}^{\top} \boldsymbol{t}^{y}_{j^*}+\log\left( \sum_{j=1}^{n_y}\exp(\boldsymbol{\Delta}^{\top} \boldsymbol{t}^{y}_j)\right)-\log(n_y)\\
	&\ge -\boldsymbol{\Delta}^{\top} \boldsymbol{t}^{y}_{j^*}+\log\left( \exp(\boldsymbol{\Delta}^{\top} \boldsymbol{t}^{y}_{j^*})\right)-\log(n_y)\\
	&= -\log(n_y),
\end{align*}
which yields a contradiction.  Hence, $\bar{\boldsymbol{t}}^{x} \not\in \boldsymbol{C}^{y}$.

Then we show that $\bar{\boldsymbol{t}}^{x} \not\in \boldsymbol{C}^{y}$ implies that $\ell_\mathrm{KL}$ is unbounded from below. Let $\bar{\boldsymbol{t}}^{x} \notin \boldsymbol{C}^{y}$. Since $\boldsymbol{C}^{y}$ and the singleton $\{\bar{\boldsymbol{t}}^{x}\}$ are disjoint compact sets, they are strongly separable, i.e., there exist some~$\boldsymbol{\Delta}^*\in\mathbb{R}^k$ and $\beta_1,\beta_2\in\mathbb{R}$ such that for all $\boldsymbol{t}\in\boldsymbol{C}^{y}$,
\begin{equation*}
	\boldsymbol{\Delta}^{*\top} \boldsymbol{t}
	\le \beta_1 < \beta_2 \le 
	\boldsymbol{\Delta}^{*\top} \bar{\boldsymbol{t}}^{x}.
\end{equation*}
This yields
\begin{align*}
	  \lim\limits_{a\to\infty}\ell_\mathrm{KL}(a\boldsymbol{\Delta}^*)
	&=\lim\limits_{a\to\infty}\left\{-a\boldsymbol{\Delta}^{*\top} \bar{\boldsymbol{t}}^{x}
	  +\log\left( \frac{1}{n_y} \sum_{j=1}^{n_y}\exp(a\boldsymbol{\Delta}^{*\top}  \boldsymbol{t}^{y}_j)\right)\right\}\\
	&\le\lim\limits_{a\to\infty}\left\{-a\boldsymbol{\Delta}^{*\top} \bar{\boldsymbol{t}}^{x}
	 +\max_{1\le j\le n_y} a\boldsymbol{\Delta}^{*\top} \boldsymbol{t}^{y}_{j}\right\}
	 \le\lim\limits_{a\to\infty}\left\{-a\beta_2+a\beta_1\right\}=-\infty,
\end{align*}
which shows that $\ell_\mathrm{KL}$ is unbounded from below.

\bigskip

It remains to argue assertion (ii). Since $\bar{\boldsymbol{t}}^x \in \relbd(\boldsymbol{C}^{y})$ is equivalent to $\bar{\boldsymbol{t}}^x \in \boldsymbol{C}^{y}\backslash \relint(\boldsymbol{C}^{y})$, the assertion follows immediately. 
\end{proof}

\subsection{Regularized KLIEP estimators}

The minimizer of $\ell_\mathrm{KL}(\boldsymbol{\Delta})$ serves as one natural estimator of $\Delta$. However, as proven in Theorem~\ref{theorem:1}, such a minimizer does not always exist, rendering the minimization of the empirical KLIEP loss an ill-posed problem. Furthermore, in the high-dimensional case ($k \gg n_y$), even if a minimizer exists, it may heavily overfit the data. Hence, we will explore some regularized estimators, with a focus on (single) norm-based penalties in this section.

\begin{theorem}\label{theorem:2}
Given two samples $\boldsymbol{X}=(\boldsymbol{X}_1,\dots,\boldsymbol{X}_{n_x})$ and $\boldsymbol{Y}=(\boldsymbol{Y}_1,\dots,\boldsymbol{Y}_{n_y})$,
let $\bar{\boldsymbol{t}}^x$ and $\boldsymbol{C}^{y}$ be defined by Equations~\eqref{eq:avetx} and \eqref{eq:Py}, respectively.  
Fix a norm $\Vert\cdot\Vert_{\ddagger}$ on $\mathbb{R}^k$ and $\lambda\ge0$. The empirical KLIEP loss with a penalty given~by
\begin{align*}
    \ell_{\mathrm{KL};\lambda,\Vert\cdot\Vert_{\ddagger}}(\boldsymbol{\Delta}) 
    \coloneqq\;& \ell_\mathrm{KL}(\boldsymbol{\Delta}) + \lambda\Vert\boldsymbol{\Delta}\Vert_{\ddagger} \\
     =\;& -\boldsymbol{\Delta}^{\top}  \bar{\boldsymbol{t}}^{x} + \log\left(\frac{1}{n_y}\sum_{j=1}^{n_y} \exp\{\boldsymbol{\Delta}^{\top}  \boldsymbol{t}^{y}_j\}\right) +\lambda\Vert\boldsymbol{\Delta}\Vert_{\ddagger}
\end{align*}
satisfies the following:
\begin{enumerate}[itemsep=-.25ex, label=(\roman*)]
	\item if $\bar{\boldsymbol{t}}^{x} \in \relint(\boldsymbol{C}^{y})$, then $\ell_{\mathrm{KL},\lambda,\Vert\cdot\Vert_{\ddagger}}$ attains a global minimum for all $\lambda\ge 0$;
	\item if $\bar{\boldsymbol{t}}^{x} \in \relbd(\boldsymbol{C}^{y})$, then $\ell_{\mathrm{KL},\lambda,\Vert\cdot\Vert_{\ddagger}}$ attains a global minimum for all $\lambda> 0$. If $\lambda=0$, then $\ell_{\mathrm{KL},\lambda,\Vert\cdot\Vert_{\ddagger}}$ is bounded from below, but does not attain a global minimum;
    \item if $\bar{\boldsymbol{t}}^{x} \notin \boldsymbol{C}^{y}$, then
    \begin{enumerate}[itemsep=-.25ex, label=(\alph*)]
    \item if $\lambda> \lambda^{\#}$, then $\ell_{\mathrm{KL},\lambda,\Vert\cdot\Vert_{\ddagger}}$ attains a global minimum;
    \item if $\lambda=\lambda^{\#}$, then $\ell_{\mathrm{KL},\lambda,\Vert\cdot\Vert_{\ddagger}}$ is bounded from below;
    \item if $\lambda<\lambda^{\#}$, then $\ell_{\mathrm{KL},\lambda,\Vert\cdot\Vert_{\ddagger}}$ is unbounded from below.
    \end{enumerate}
Here, $\lambda^{\#}$ is given by
 \begin{equation}\label{lambdastar}
	    \lambda^{\#}\coloneqq \min\limits_{\boldsymbol{t}\in\boldsymbol{C}^{y}}\Vert \bar{\boldsymbol{t}}^{x}-\boldsymbol{t}\Vert_{\#}
	\end{equation} where $\Vert\cdot\Vert_{\#}$ is the norm dual to $\Vert\cdot\Vert_{\ddagger}$ and $\boldsymbol{C}^{y}=\conv(\{\boldsymbol{t}^{y}_1,\ldots,\boldsymbol{t}^{y}_{n_y}\})$ as in Theorem \ref{theorem:1}, i.e., $\lambda^{\#}$ is the $\Vert\cdot\Vert_{\#}$-norm distance of $\bar{\boldsymbol{t}}^{x}$ to the polytope spanned by the $\boldsymbol{t}^{y}_1,\ldots,\boldsymbol{t}^{y}_{n_y}$.
\end{enumerate}
\end{theorem}

\begin{proof}[Proof of Theorem~\ref{theorem:2}]
In view of Theorem \ref{theorem:1}, assertion (i) follows immediately.  The second part of assertion (ii) is also obvious, while the first part follows from Lemma~\ref{lem:coercive} easily.

It remains to show assertion (iii). First, we give an alternative representation for $\lambda^{\#}$.
Denote $\bar{\mathbb{B}}^k_{\Vert\cdot\Vert_{\ddagger}}=\{x\in\mathbb{R}^k:\Vert x\Vert_{\ddagger}\le1\}$ the closure of the unit ball with respect to $\Vert\cdot\Vert_{\ddagger}$.  
We find
\begin{align*}
        \lambda^{\#}
         =\min\limits_{\boldsymbol{t}\in\boldsymbol{C}^{y}}\Vert \bar{\boldsymbol{t}}^{x}-\boldsymbol{t}\Vert_{\#}
        &=\min\limits_{\boldsymbol{t}\in\boldsymbol{C}^{y}}\max\limits_{{\boldsymbol{\Delta}}\in \bar{\mathbb{B}}^d_{\Vert\cdot\Vert_{\ddagger}}}{\boldsymbol{\Delta}}^{\top} (\bar{\boldsymbol{t}}^{x} - \boldsymbol{t})\\
        &=\max\limits_{{\boldsymbol{\Delta}}\in \bar{\mathbb{B}}^d_{\Vert\cdot\Vert_{\ddagger}}}\min\limits_{\boldsymbol{t}\in\boldsymbol{C}^{y}}{\boldsymbol{\Delta}}^{\top} (\bar{\boldsymbol{t}}^{x} - \boldsymbol{t})
         =\max\limits_{{\boldsymbol{\Delta}}\in \mathbb{S}^{d-1}_{\Vert\cdot\Vert_{\ddagger}}}\min\limits_{\boldsymbol{t}\in\boldsymbol{C}^{y}}{\boldsymbol{\Delta}}^{\top} (\bar{\boldsymbol{t}}^{x} - \boldsymbol{t}).
\end{align*}
Here, we use the Minimax Theorem \citep[Propistion~5.4.4]{MR3444832} as stated below.  
\begin{lemma}[Minimax Theorem]\label{minimax}
    Let $X\subseteq\mathbb{R}^n$ and $Y\subseteq\mathbb{R}^m$ be compact convex sets. Let $f:X\times Y\to\mathbb{R}$ a bilinear form. Then we have\begin{equation}
        \max\limits_{x\in X}\min\limits_{y\in Y}f(x,y)=\min\limits_{y\in Y}\max\limits_{x\in X}f(x,y).
    \end{equation}
\end{lemma}
If a linear function attains a minimum over the polytope set $\boldsymbol{C}^{y}$, which has at least one extreme point, then the function must attain a minimum at some extreme point of $\boldsymbol{C}^{y}$ \citep[page~701]{MR3444832}. Therefore, we further deduce
\begin{equation}
\lambda^{\#}
         =\max\limits_{{\boldsymbol{\Delta}}\in \mathbb{S}^{d-1}_{\Vert\cdot\Vert_{\ddagger}}}\min\limits_{1\le j\le n_y}{\boldsymbol{\Delta}}^{\top} (\bar{\boldsymbol{t}}^{x} - \boldsymbol{t}^{y}_j)
         =-\min\limits_{{\boldsymbol{\Delta}}\in\mathbb{S}^{d-1}_{\Vert\cdot\Vert_{\ddagger}}}\max\limits_{1\le j\le n_y}\boldsymbol{{\Delta}}^{\top} \boldsymbol{u}_j.
         \label{lambdastar2}
\end{equation}
The obtained representation for $\lambda^{\#}$ enables us to show the claim. For all $\boldsymbol{\Delta}\ne0$, writing
$
\tilde{\boldsymbol{\Delta}}\coloneqq {\boldsymbol{\Delta}}/{\Vert\boldsymbol{\Delta}\Vert}
$,
we have
\begin{align}
    \ell_\mathrm{KL}(\boldsymbol{\Delta})&=\log\left( \frac{1}{n_y} \sum_{j=1}^{n_y}\exp(\boldsymbol{\Delta}^{\top}  \boldsymbol{u}_j)\right) \notag\\
    & = \log\left( \frac{1}{n_y} \sum_{j=1}^{n_y}\exp(\Vert\boldsymbol{\Delta}\Vert_{\ddagger}\tilde{\boldsymbol{\Delta}}^{\top}  \boldsymbol{u}_j)\right) \notag\\
    &\ge \log\left( \frac{1}{n_y} \max_{1\le j\le n_y}\exp(\Vert\boldsymbol{\Delta}\Vert_{\ddagger}\tilde{\boldsymbol{\Delta}}^{\top} \boldsymbol{u}_{j})\right) \notag\\
    & = \max_{1\le j\le n_y}\Vert\boldsymbol{\Delta}\Vert_{\ddagger}\tilde{\boldsymbol{\Delta}}^{\top} \boldsymbol{u}_{j} - \log(n_y) \notag\\
    &\ge-\lambda^{\#}\Vert\boldsymbol{\Delta}\Vert_{\ddagger} - \log(n_y).
    \label{lowerbound}
\end{align}
We then have the following results:
\begin{enumerate}[itemsep=-.25ex, label=(\alph*)]
\item If $\lambda>\lambda^{\#}$, we may add the penalty term $\lambda\Vert\boldsymbol{\Delta}\Vert_{\ddagger}$ to both sides of the bound in \eqref{lowerbound} to obtain that for all $\boldsymbol{\Delta}\in\boldsymbol{U}$,
\begin{equation}
    \ell_{\mathrm{KL};\lambda,\Vert\cdot\Vert_{\ddagger}}(\boldsymbol{\Delta})\ge (\lambda-\lambda^{\#})\Vert\boldsymbol{\Delta}\Vert_{\ddagger}-\log(n_y),
\label{lowerbound2}
\end{equation}
which yields that $\ell_{\mathrm{KL};\lambda,\Vert\cdot\Vert_{\ddagger}}$ is continuous and coercive and therefore attains a global minimum.   
\item If $\lambda=\lambda^{\#}$, by~\eqref{lowerbound2}, we see that $\ell_{\mathrm{KL};\lambda,\Vert\cdot\Vert_{\ddagger}}$ is lower bounded on $\boldsymbol{U}$ and hence on~$\mathbb{R}^k$.
\item If $\lambda<\lambda^{\#}$, then fix some $\boldsymbol{\Delta}^*\in\mathbb{S}^{d-1}_{\Vert\cdot\Vert_{\ddagger}}$ which attains the minimum in \eqref{lambdastar2}. We have
\begin{align*}
    \lim\limits_{a\to\infty}\ell_{\mathrm{KL};\lambda,\Vert\cdot\Vert_{\ddagger}}(a\boldsymbol{\Delta}^*)
    &=\lim\limits_{a\to\infty}\left\{\log\left( \frac{1}{n_y} \sum_{j=1}^{n_y}\exp(a\boldsymbol{\Delta}^{*\top}  \boldsymbol{u}_j)\right) + \lambda\Vert a\boldsymbol{\Delta}^*\Vert_{\ddagger}\right\} \\
    &\le\lim\limits_{a\to\infty}\left\{\log\left( \max_{1\le j\le n_y}\exp(a\boldsymbol{\Delta}^{*\top}  \boldsymbol{u}_j)\right) + a\lambda\right\} \\
    &=\lim\limits_{a\to\infty}\left\{\max_{1\le j\le n_y} a\boldsymbol{\Delta}^{*\top}  \boldsymbol{u}_j + a\lambda\right\} \\
    &=\lim\limits_{a\to\infty}\left\{a(\lambda-\lambda^{\#})\right\} =-\infty,
\end{align*}
and thus $\ell_{\mathrm{KL};\lambda,\Vert\cdot\Vert_{\ddagger}}$ is unbounded from below.
\end{enumerate}
This completes the proof.
\end{proof}

\begin{ex}[Dual Norms]
For $p\in [1,\infty]$, the $\ell_p$-norm of vector $\boldsymbol{x}\in\mathbb{R}^k$ is
$\|\boldsymbol{x} \|_{p}~:=~\left(\sum _{i=1}^{k}\left|x_{i}\right|^{p}\right)^{1/p}$.
     For $q\in[1,\infty]$ with $\frac{1}{p}+\frac{1}{q}=1$, the norm dual to $\ell_p$ is $\ell_q$ \citep[page~331]{MR2978290}. For example, this implies that the norm dual to $\ell_1$ is $\ell_\infty$, and $\ell_2$ has itself as the dual norm. 
\end{ex}

\section{Differential network analysis}
\label{sec:diffnets}

We now explore the existence issues in KLIEP estimation in the context of differential network analysis and, specifically, inferring differences between graphical models in a Gaussian case.

\subsection{Pairwise interactions and Gaussian graphical models} 
 
Consider the following family of pairwise interaction models on $\mathbb{R}^m$, which is a special case of the exponential family with densities of the form \eqref{eq:expfam}:
\begin{equation}\label{pairwise}
    f(\boldsymbol{x};\boldsymbol{\theta}) = \frac{1}{Z(\boldsymbol{\theta})} 
    \exp\left(
    \sum_{u,v=1, v \ge u}^{m} \theta_{uv} t_{uv}(x_u,x_v)\right), \quad \boldsymbol{x}\in\mathbb{R}^m.
\end{equation}
Here, the natural parameter vector is $\boldsymbol{\theta}=(\theta_{uv})_{u\le v}$, and the sufficient statistic $\mathbf{t}$ has coordinates $t_{uv}(x_u,x_v)$ that depend only on pairs of coordinates of $\boldsymbol{x}$.  
Such pairwise interaction models were studied in \citet{MR3662445}, with Gaussian graphical models and Ising models serving as the two most prominent examples.  These two cases have different dominating measures but share the form of the statistics with $\textbf{t}_{uv}(x_u,x_v)$.  We, thus, present simulation results focused on Gaussian models.

Gaussian graphical models are based on multivariate normal  distributions with density
\[
f(\boldsymbol{x};\boldsymbol{\Theta}) = 
    \frac{\mathrm{det}(\boldsymbol{\Theta})^{1/2}}{(2\pi)^{m/2}}
    \exp\left(-\frac12\boldsymbol{x}^{\top}\boldsymbol{\Theta}\boldsymbol{x}\right), \quad \boldsymbol{x}\in\mathbb{R}^m,
\]
where $\boldsymbol{\Theta}\in\mathbb{R}^{m\times m}$ is the precision matrix and $\mathrm{det}(\cdot)$ denotes the determinant.
The model is an instance of a pairwise interaction model as specified in~\eqref{pairwise} with $\boldsymbol{\theta}$ being the half-vectorization of~$\boldsymbol{\Theta}$ and the sufficient statistic $\mathbf{t}$ having components
\begin{align*}
t_{uv}(x_u,x_v) &= \begin{cases} -x_u^2/2, & u=v, \\ -x_ux_v, & u\neq v,\end{cases}
\end{align*}
Sparsity patterns in the precision matrix may be encoded in the (undirected) concentration graph $G = (V, E)$, which has vertex set $V=\{1,\dots,m\}$ and edge $(i,j)\in E$ if and only if $\theta_{ij}\not=0$.

\subsection{Simulation study}
\label{subsec:sims}

For our simulations, we take up two often considered structures for graphical models, namely, lattice and random graphs.  A lattice graph has its nodes structured in the form of a two-dimensional grid, which here takes a square as shown in Figure~\ref{fig:lattice_ex}. 
The focus of our study is the relationship between
the threshold $\lambda^{\#}$ defined in Equation~\eqref{lambdastar}, in particular, the~$\ell_\infty$ distance of $\bar{\boldsymbol{t}}^{x}$ to the polytope $\boldsymbol{C}^{y}$ in the case when the regularized KLIEP loss $\ell_{\mathrm{KL},\lambda,\Vert\cdot\Vert_{1}}$ is concerned, and
the regularization parameter $\lambda_{\rm Liu}=2.5\sqrt{\frac{\log m}{n_p}}$ suggested by \citet{MR3662445}. 
When $\lambda_{\rm Liu}$ is smaller than $\lambda^{\#}$, the global minimum of $\ell_{\mathrm{KL},\lambda,\|\cdot\|_{1}}$ does not exist and theoretical guarantees from \citet{MR3662445} become vacuous.

\paragraph{Lattice graph.}

In our simulations, the samples $\boldsymbol{X}=(\boldsymbol{X}_1,\dots,\boldsymbol{X}_{n_x})$ and $\boldsymbol{Y}=(\boldsymbol{Y}_1,\dots,\boldsymbol{Y}_{n_y})$ are independent and drawn i.i.d.~from Gaussian distributions $P$ and $Q$, respectively.  Their densities are parametrized as
\begin{align}
\label{eq:p}
p(\boldsymbol{x})  &\propto \exp\left( -\frac12\left\{\sum_{u=1}^{m}\theta_0 x_u^2 + \sum_{(u,v)\in ED_p}\theta_1 x_u x_v \right\}\right), \\
\label{eq:q}
q(\boldsymbol{x})  &\propto \exp\left( -\frac12\left\{\sum_{u=1}^{m}\theta_0 x_u^2 + \sum_{(u,v)\in ED_q\backslash ED_{\rm rdm}}\theta_1 x_u x_v + \sum_{(u,v)\in ED_{\rm rdm}}\theta_1^* x_u x_v \right\}\right),
\end{align}
respectively. In~\eqref{eq:p} and \eqref{eq:q},
$ED_p=ED_q$ are the edge sets of a lattice graph with $m$ nodes, and $ED_{\rm rdm}$ denotes a set of $d$ randomly picked edges from $ED_q$.

\begin{figure}[t]
\centering
\resizebox{2in}{2in}{
\begin{tikzpicture}[-,>=stealth',shorten >=1pt,auto,node distance=2cm,semithick]
  \tikzstyle{every state}=[fill=gray!30,draw=black,text=black]

  \node[state,minimum size=1cm]                      (A)                     {$A_{1}$};
  \node[align=center,state,minimum size=1cm]         (B) [right of=A]        {$A_{2}$};
  \node[align=center,state,minimum size=1cm]         (C) [right of=B]        {$A_{3}$};
  \node[align=center,state,minimum size=1cm]         (D) [right of=C]        {$A_{4}$};

  \node[align=center,state,minimum size=1cm]         (F) [below of=A]        {$A_{5}$};
  \node[align=center,state,minimum size=1cm]         (G) [below of=B]        {$A_{6}$};
  \node[align=center,state,minimum size=1cm]         (H) [below of=C]        {$A_{7}$};
  \node[align=center,state,minimum size=1cm]         (I) [below of=D]        {$A_{8}$};

  \node[align=center,state,minimum size=1cm]         (K) [below of=F]        {$A_{9}$};
  \node[align=center,state,minimum size=1cm]         (L) [below of=G]        {$A_{10}$};
  \node[align=center,state,minimum size=1cm]         (M) [below of=H]        {$A_{11}$};
  \node[align=center,state,minimum size=1cm]         (N) [below of=I]        {$A_{12}$};

  \node[align=center,state,minimum size=1cm]         (P) [below of=K]        {$A_{13}$};
  \node[align=center,state,minimum size=1cm]         (Q) [below of=L]        {$A_{14}$};
  \node[align=center,state,minimum size=1cm]         (R) [below of=M]        {$A_{15}$};
  \node[align=center,state,minimum size=1cm]         (S) [below of=N]        {$A_{16}$};

  \draw (A) -- (B)   (B) -- (C)   (C) -- (D)
        (F) -- (G)   (G) -- (H)   (H) -- (I)
        (K) -- (L)   (L) -- (M)   (M) -- (N)
        (P) -- (Q)   (Q) -- (R)   (R) -- (S)

        (A) -- (F)   (F) -- (K)   (K) -- (P)
        (B) -- (G)   (G) -- (L)   (L) -- (Q)
        (C) -- (H)   (H) -- (M)   (M) -- (R)
        (D) -- (I)   (I) -- (N)   (N) -- (S);

\end{tikzpicture}
}
\caption{Example of a lattice graph when $m=4^2$.}\label{fig:lattice_ex}
\end{figure}
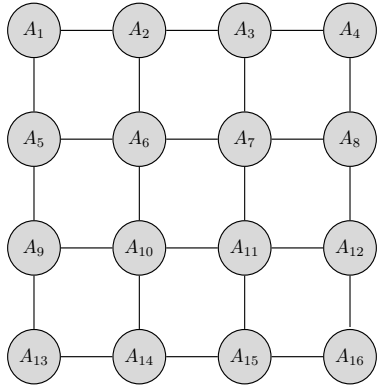

The plots in Figure~\ref{figure:lattice_c4} report the threshold $\lambda^{\#}$ based on $100$ simulations for sample sizes $n_q=1000$ (fixed) and $n_p/\log m\in\{50,100,150,200\}$, in dimensions $m\in\{12^2, 16^2, 20^2\}$, when the number of changed edges $d\in\{10,20,40,80\}$, for parameter values $\theta_0=2$, $\theta_1=0.4$, $\theta_1^*=-0.4$.
For comparison, the regularization parameter $\lambda_{\rm Liu}=2.5\sqrt{\frac{\log m}{n_p}}$ suggested by \citet{MR3662445} is shown as a horizontal red line, and the proportion of $\lambda^{\#}$ larger than $\lambda_{\rm Liu}$ is reported in the top margin of each plot. 
Figure~\ref{figure:lattice_c45} shows the results changing the parameter values to $\theta_0=2$, $\theta_1=0.4$, $\theta_1^*=-0.8$.

As can be expected, Figure~\ref{figure:lattice_c4} shows an increase in $\lambda^{\#}$ as the number of changed edges $d$ increases, keeping all other variables fixed.  This figure echoes the simulations performed in \citet{MR3662445}, where the performance of the regularization parameter $\lambda_{\rm Liu}=2.5\sqrt{\frac{\log m}{n_p}}$ in change detection is satisfying. Indeed, in this case, $\lambda_{\rm Liu}$ is not lower bounded by $\lambda^{\#}$ only when $d=80$ and $n_p=200 \log m$.

Comparing Figures~\ref{figure:lattice_c4} and~\ref{figure:lattice_c45}, we observe that the increased magnitude of the difference parameter $\theta_1^*$ leads to an increase in $\lambda^{\#}$. In Figure~\ref{figure:lattice_c45}, even for $d=20$, there is a non-negligible probability that $\lambda_{\rm Liu}$ is smaller than $\lambda^{\#}$. In this case, one would have to find a larger regularization parameter to infer differences between two Gaussian graphical models.

\begin{figure}
\centering
\includegraphics[width=\textwidth]{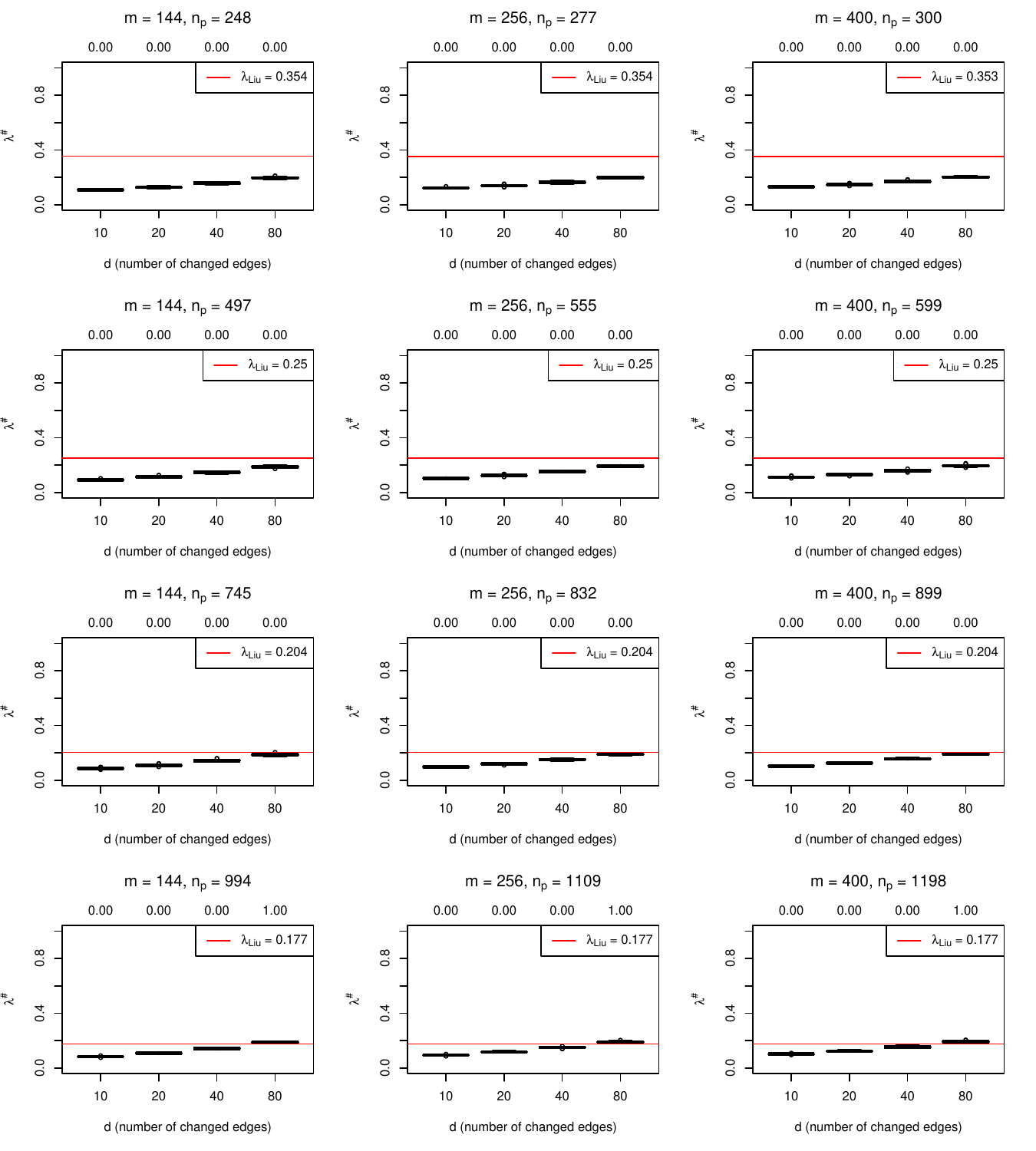}
\caption{
The comparison between the threshold $\lambda^{\#}$ and the tuning parameter $\lambda_{\rm Liu}=2.5\sqrt{\frac{\log m}{n_p}}$ suggested by \citet{MR3662445} under Gaussian graphical model in lattice structure, 
based on~100 replications for sample sizes $n_q=1000$ (fixed) and $n_p/\log m\in\{50,100,150,200\}$, in dimensions $m\in\{12^2, 16^2, 20^2\}$, when the number of changed edges $d\in\{10,20,40,80\}$, for parameter values $\theta_0=2, \theta_1=0.4, \theta_1^*=-0.4$.
}\label{figure:lattice_c4}
\end{figure}

\begin{figure}
\centering
\includegraphics[width=\textwidth]{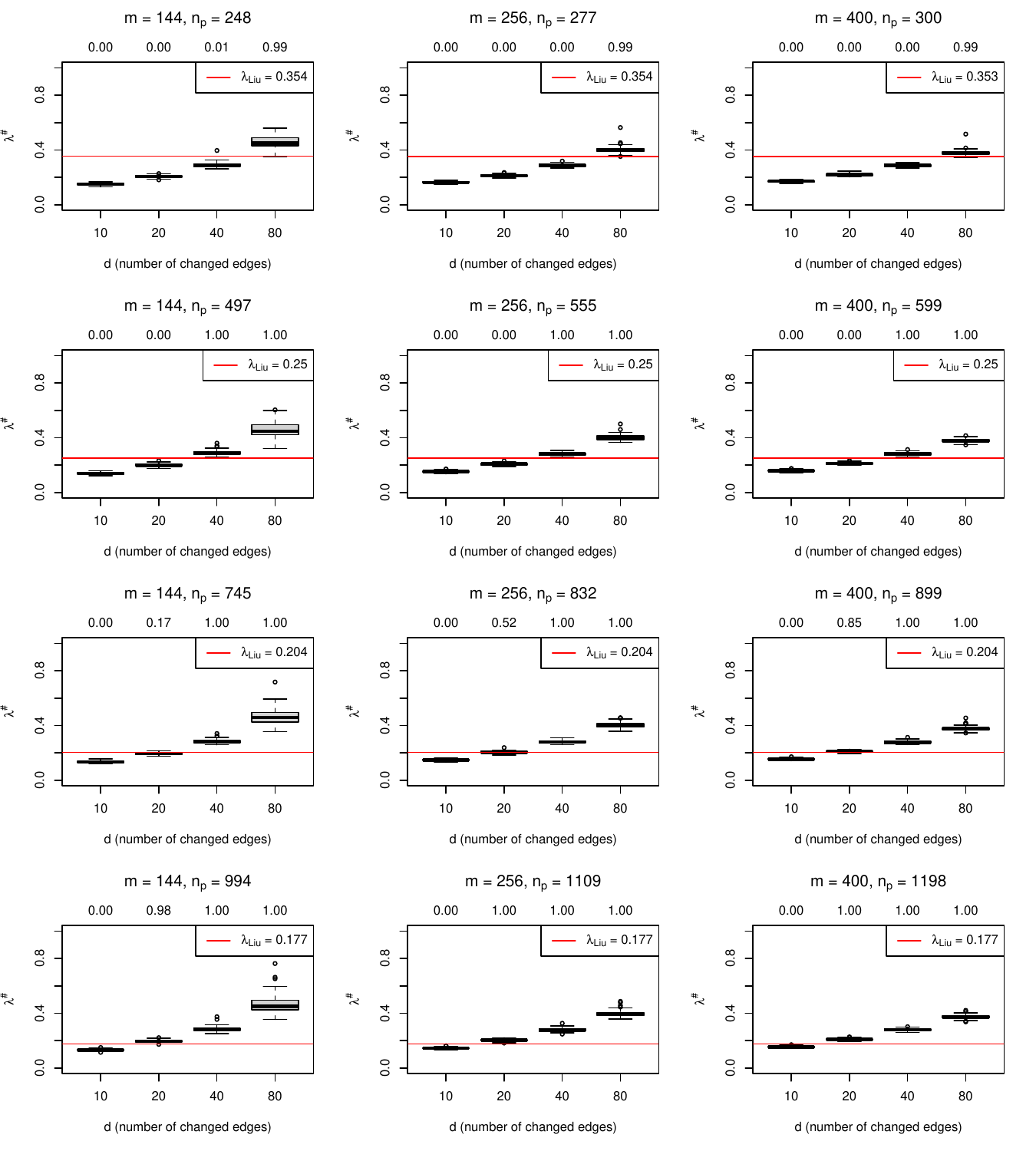}
\caption{
The comparison between the threshold $\lambda^{\#}$ and the tuning parameter $\lambda_{\rm Liu}=2.5\sqrt{\frac{\log m}{n_p}}$ suggested by \citet{MR3662445} under Gaussian graphical model in lattice structure, 
based on~100 replications for sample sizes $n_q=1000$ (fixed) and $n_p/\log m\in\{50,100,150,200\}$, in dimensions $m\in\{12^2, 16^2, 20^2\}$, when the number of changed edges $d\in\{10,20,40,80\}$, for parameter values $\theta_0=2, \theta_1=0.4, \theta_1^*=-0.8$.
}\label{figure:lattice_c45}
\end{figure}

\paragraph{Random graph.}

Next, we study the relationship between $\lambda^{\#}$ and $\lambda_{\rm Liu}=2.5\sqrt{\frac{\log m}{n_p}}$ when $ED_p=ED_q$ are the edge sets of an Erd\H{o}s-Renyi {\it random graph} with $m$ nodes. The probability of drawing an edge between two arbitrary vertices in the random graph is set to be $0.4$. All other settings remain the same as in the previous case of lattice graphs. 

We observe from Figures~\ref{figure:random_c4} and~\ref{figure:random_c45} that $\lambda_{\rm Liu}$ is smaller than the threshold $\lambda^{\#}$ with a higher probability in random structures compared to lattice structures.  In this case, even when the number of changed edges is relatively small, such as $10$, the non-existence issue of a global minimum for the regularized KLIEP loss with the tuning parameter $\lambda_{\rm Liu}$ can occur with a probability as high as~$90\%$.

\begin{figure}
\centering
\includegraphics[width=\textwidth]{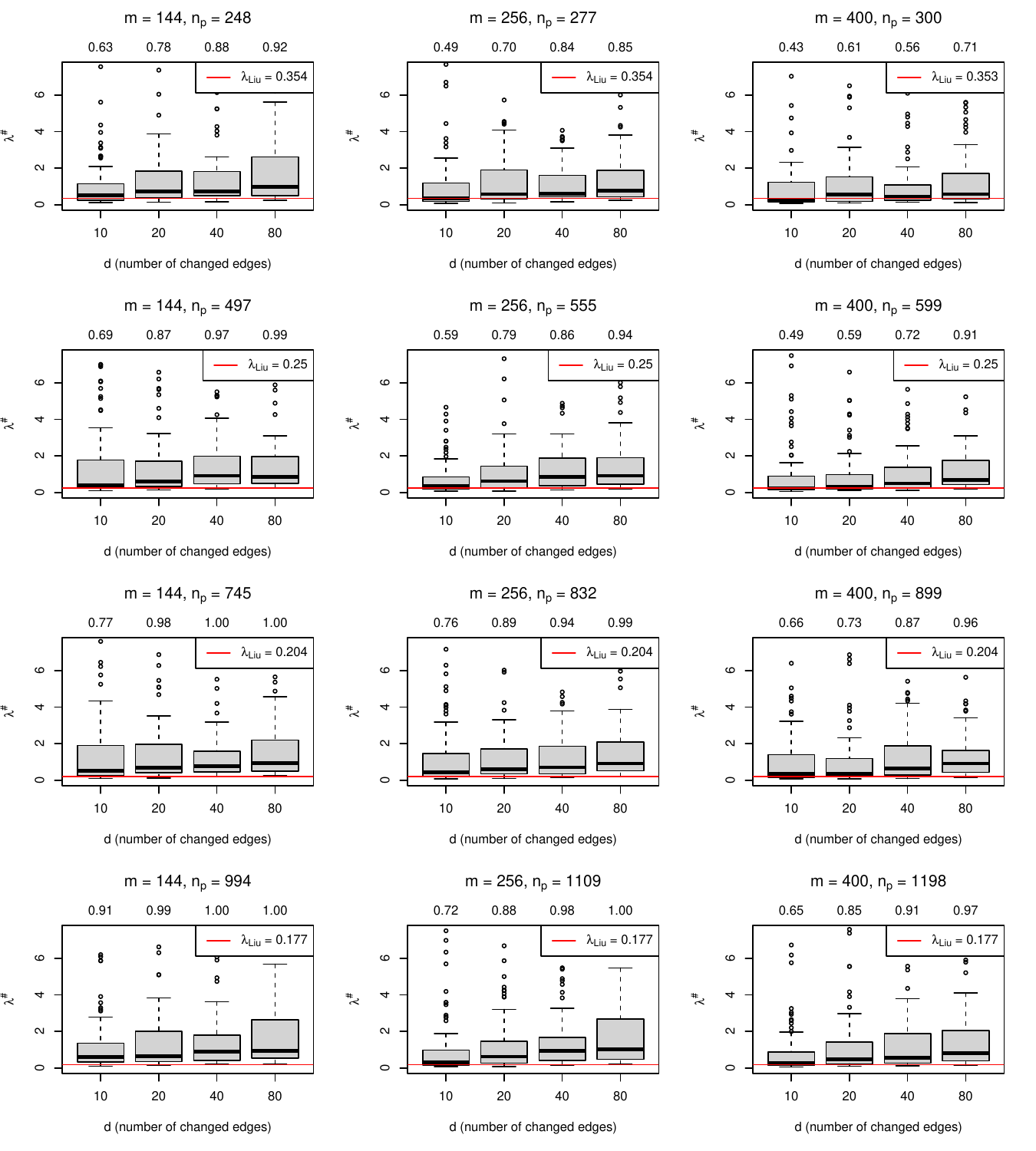}
\caption{
The comparison between the threshold $\lambda^{\#}$ and the tuning parameter $\lambda_{\rm Liu}=2.5\sqrt{\frac{\log m}{n_p}}$ suggested by \citet{MR3662445} under Gaussian graphical model in random structure, 
based on~100 replications for sample sizes $n_q=1000$ (fixed) and $n_p/\log m\in\{50,100,150,200\}$, in dimensions $m\in\{12^2, 16^2, 20^2\}$, when the number of changed edges $d\in\{10,20,40,80\}$, for parameter values $\theta_0=2, \theta_1=0.4, \theta_1^*=-0.4$.
}\label{figure:random_c4}
\end{figure}

\begin{figure}
\centering
\includegraphics[width=\textwidth]{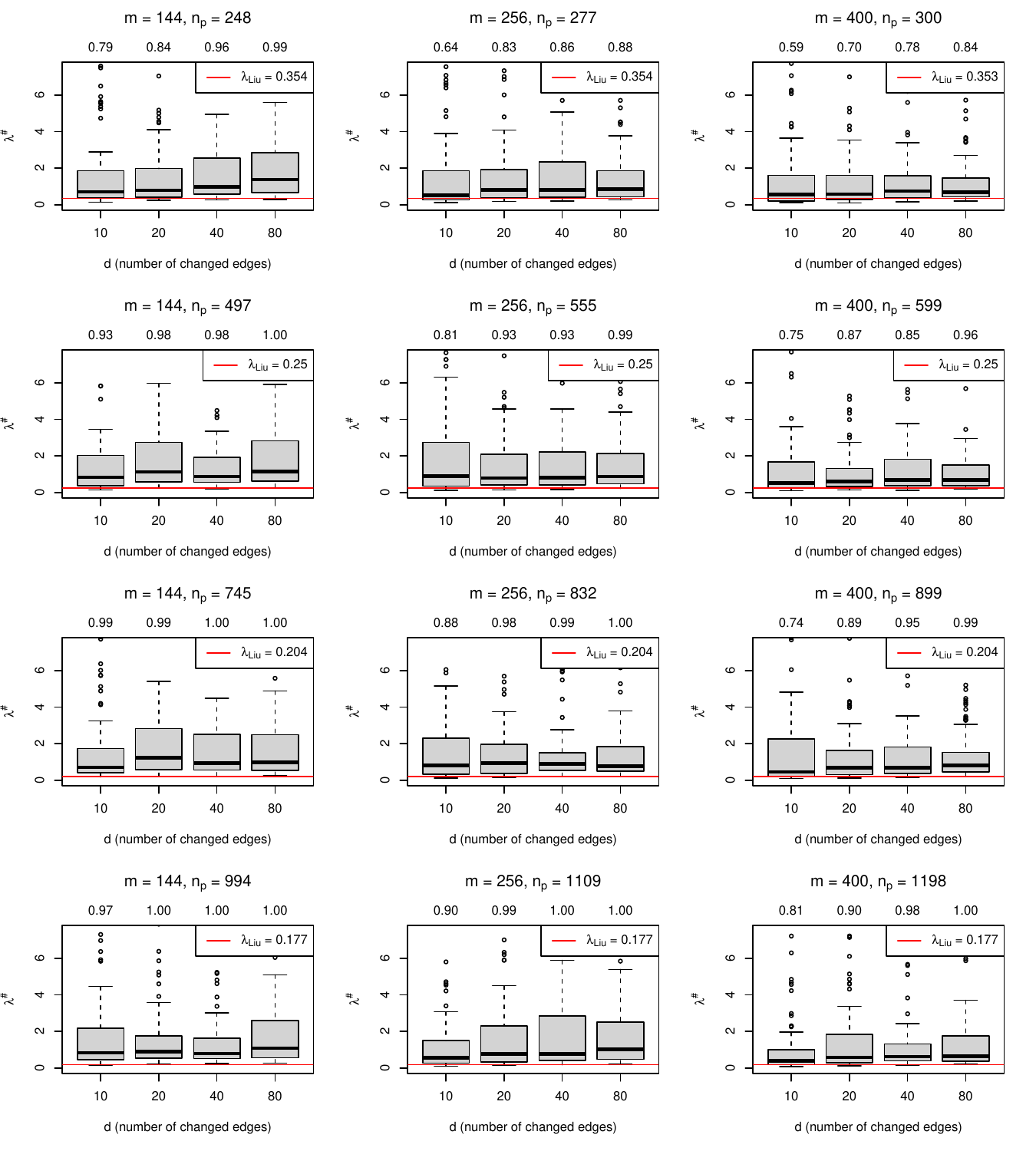}
\caption{
The comparison between the threshold $\lambda^{\#}$ and the tuning parameter $\lambda_{\rm Liu}=2.5\sqrt{\frac{\log m}{n_p}}$ suggested by \citet{MR3662445} under Gaussian graphical model in random structure, 
based on~100 replications for sample sizes $n_q=1000$ (fixed) and $n_p/\log m\in\{50,100,150,200\}$, in dimensions $m\in\{12^2, 16^2, 20^2\}$, when the number of changed edges $d\in\{10,20,40,80\}$, for parameter values $\theta_0=2, \theta_1=0.4, \theta_1^*=-0.8$.
}\label{figure:random_c45}
\end{figure}

\bigskip

In summary, the non-existence issue of a global minimum for the regularized KLIEP when the tuning parameter $\lambda_{\rm Liu}$ is employed can occur with significant probability, even when $d$ is small or moderate. 
Moreover, the non-existence issue does not diminish as the sample size $n_p$ increases when $n_q$ is fixed but sufficiently large. 
This is implicitly suggested by Theorem 1 of \citet{MR3662445}, as $n_p / \log m$ is required to be larger than the rate of $d^2$. 
This could also be attributed to the fact that, in our simulations, $n_q$ is not sufficiently large --- specifically, $n_q = 0.01 n_p^2$ was used in some experiments in \citet{MR3662445}, as suggested by Theorem 1. However, this implies that $n_q$ would need to exceed $10,000$ when $n_p > 1,000$, and such an unbalanced design might rarely occur in real-world practice.

\section{Discussion}\label{sec:dis}

Unlike the empirical squared error loss featured in lasso regression, the empirical KLIEP loss is not always positive. As a result, as shown in Theorem~\ref{theorem:2}, even when a norm-based penalty is added, the global minimum may not exist unless the tuning parameter is sufficiently large. Moreover, the simulations in Section~\ref{subsec:sims} confirm in concrete settings, which parallel that of numerical experiments in the literature, that the (optimal) tuning parameter suggested by \citet{MR3662445} can be too small to ensure the existence of the regularized KLIEP estimator with $\ell_1$-penalty. 

To overcome the possible existence issues for KLIEP, it is desirable to design modifications of the KLIEP loss that avoid unboundedness issues.  As a possible step in this direction, 
we suggest considering a modification that augments the KLIEP loss with an elastic penalty, in which case the minimizer always exists.  This suggestion is related to a proposal of \citet{MR3960930} for high-dimensional one-sample problems approached via score matching.

\begin{theorem}\label{theorem:3}
The empirical KLIEP loss with an elastic net penalty given by
\begin{align*}
    \ell_{\mathrm{KL};\mathrm{els}}(\boldsymbol{\Delta}) 
    \coloneqq\;& \ell_\mathrm{KL}(\boldsymbol{\Delta}) + \lambda_1\Vert\boldsymbol{\Delta}\Vert_{1} + \lambda_2\Vert\boldsymbol{\Delta}\Vert^2 \\
     =\;& -\boldsymbol{\Delta}^{\top}  \bar{\boldsymbol{t}}^{x} + \log\left(\frac{1}{n_y}\sum_{j=1}^{n_y} \exp\{\boldsymbol{\Delta}^{\top}  \boldsymbol{t}^{y}_j\}\right) + \lambda_1\Vert\boldsymbol{\Delta}\Vert_{1} + \lambda_2\Vert\boldsymbol{\Delta}\Vert^2
\end{align*}
with $\lambda_1\ge 0,\lambda_2>0$ always attains a global minimum.
\end{theorem}

\begin{proof}[Proof of Theorem~\ref{theorem:3}]
It suffices to prove $\ell_{\mathrm{KL};\mathrm{els}}(\boldsymbol{\Delta})$ is continuous and coercive.
The continuity of $\ell_{\mathrm{KL};\mathrm{els}}(\boldsymbol{\Delta})$ is obvious. It also holds that
\begin{align*}
 \ell_{\mathrm{KL};\mathrm{els}}(\boldsymbol{\Delta}) 
 &\ge -\lVert\boldsymbol{\Delta}\rVert  \lVert\bar{\boldsymbol{t}}^{x}\rVert
      -\lVert\boldsymbol{\Delta}\rVert  \max_{1\le j\le n_y}\lVert \boldsymbol{t}^{y}_j\rVert
      +\lambda_2 \lVert\boldsymbol{\Delta}\rVert^2
\end{align*}
and thus
$
\lim\limits_{\Vert\boldsymbol{\Delta}\Vert\to\infty} 
            \ell_{\mathrm{KL};\mathrm{els}}(\boldsymbol{\Delta}) = +\infty.
$
This completes the proof.
\end{proof}

A downside of working with the elastic net is the need to set tuning parameters, here, $\lambda_1$ and~$\lambda_2$. However, we conjecture that good estimation results can be obtained from a fixed small choice of $\lambda_2$, with a value that is set merely as a function of the sample sizes and dimensions. We leave a detailed numerical and theoretical study of elastic net methods for KLIEP to future work.

\section*{Acknowledgements}
This project has received funding from the European Research Council (ERC) under the European Union's Horizon 2020 research and innovation programme (grant agreement No 883818).

\bibliographystyle{apalike-three}
\bibliography{biblio}

\end{document}